\documentclass{amsart}
\usepackage[all]{xy}
\usepackage{amssymb,latexsym, amscd, amsthm, amsmath, upgreek, bm, dsfont, mathrsfs}
\newtheorem{theorem}{Theorem}[section]
\newtheorem{lemma}[theorem]{Lemma}

\theoremstyle{definition}
\newtheorem{definition}[theorem]{Definition}

\theoremstyle{remark}

\newcommand{\C}{\mathbb C}

\newcommand{\G}{\mathbb G}
\newcommand{\Z}{\mathbb Z}
\newcommand{\R}{\mathbb R}

\newcommand{\SSS}{\mathbb S}

\newcommand{\trnm}[1]{\mbox{#1}}
\newcommand{\ncd}{\newcommand}
\ncd{\nn}{\nonumber}
\ncd{\ba}{\begin{array}}
\ncd{\ea}{\end{array}}
\ncd{\be}{\begin{equation}}
\ncd{\ee}{\end{equation}}
\ncd{\bea}{\begin{eqnarray}}
\ncd{\eea}{\end{eqnarray}}
\ncd{\ga}{\alpha}
\ncd{\gb}{\beta}
\ncd{\GG}{\Gamma}
\ncd{\gl}{\lambda}
\ncd{\GL}{\Lambda}
\ncd{\grg}{\gamma}
\ncd{\go}{\omega}
\ncd{\GO}{\Omega}
\ncd{\gr}{\rho}
\ncd{\nl}{\nabla}
\ncd{\gre}{\epsilon}
\ncd{\td}{\widetilde}
\ncd{\st}{\longrightarrow}
\ncd{\stk}{\rightarrow}
\ncd{\na}{\longmapsto}
\ncd{\fa}{\forall}
\ncd{\cc}{\circ}
\ncd{\za}{\subset}
\ncd{\wtw}{\Longleftrightarrow}
\ncd{\pcz}{\partial}
\ncd{\gz}{\zeta}
\ncd{\gt}{\theta}
\ncd{\gvt}{\vartheta}
\ncd{\gs}{\sigma}
\ncd{\GS}{\Sigma}
\ncd{\cl}{{\cal L}}
\ncd{\gd}{\delta}
\ncd{\GD}{\Delta}
\ncd{\ld}{\trnm{Ldiff}}
\ncd{\End}{\trnm{End}}
\ncd{\cj}{{\cal J}}
\ncd{\ce}{{\cal E}}
\ncd{\cf}{{\cal F}}
\ncd{\tr}{{\trnm Tr}}
\ncd{\bgw}{{\bigwedge}}
\ncd{\im}{\trnm{im}}
\ncd{\szesc}{{\times _6}}
\ncd{\siedem}{{\times _7}}
\ncd{\ppp}{{\times _p}}
\ncd{\la}{\langle}
\ncd{\ra}{\rangle}
\ncd{\bra }{\langle}
\ncd{\ket}{\rangle}
\ncd{\mnn}{ M_{n_1,\ldots ,n_N}}
\ncd{\pcuj}{\frac{\pcz}{\pcz u_1}}
\ncd{\pcucj}{\frac{\pcz}{\pcz u^*_1}}
\ncd{\pcud}{\frac{\pcz}{\pcz u_2}}
\ncd{\pcucd}{\frac{\pcz}{\pcz u^*_2}}
\ncd{\pcut}{\frac{\pcz}{\pcz u_3}}
\ncd{\pcuct}{\frac{\pcz}{\pcz u^*_3}}
\ncd{\pcum}{\frac{\pcz}{\pcz u_M}}
\ncd{\pcucm}{\frac{\pcz}{\pcz u^*_M}}
\ncd{\pcumj}{\frac{\pcz}{\pcz u_{M-1}}}
\ncd{\pcucmj}{\frac{\pcz}{\pcz u^*_{M-1}}}
\ncd{\pcujq}{\frac{\pcz ^2}{\pcz u_1 \pcz u^*_1}}
\ncd{\pcudq}{\frac{\pcz ^2}{\pcz u_2 \pcz u^*_2}}
\ncd{\jedn}{\hat{h}}
\ncd{\pczxk}{\frac{\pcz}{\pcz x_k}}
\ncd{\pcuk}{\frac{\pcz}{\pcz u_k}}
\ncd{\pcul}{\frac{\pcz}{\pcz u_l}}
\ncd{\pcuck}{\frac{\pcz}{\pcz u^*_k}}
\ncd{\pcucl}{\frac{\pcz}{\pcz u^*_l}}
\ncd{\pczk}{\frac{\pcz}{\pcz z_k}}
\ncd{\pczl}{\frac{\pcz}{\pcz z_l}}
\ncd{\pczck}{\frac{\pcz}{\pcz z^*_k}}
\ncd{\pczcl}{\frac{\pcz}{\pcz z^*_l}}
\ncd{\snk}{\sum_{k=1}^n}
\ncd{\sn}{{N \choose 2}}
\ncd{\gk}{\overline}
\ncd{\id}{\trnm{id}}

\begin{document}

\title[Mixed Hodge structures and Weierstrass $\sigma$-function]{Mixed Hodge structures and Weierstrass $\sigma$-function}

\author[G. Banaszak]{Grzegorz Banaszak*}
\address{Department of Mathematics and Computer Science, Adam Mickiewicz University,
Pozna\'{n} 61-614, Poland}
\email{banaszak@amu.edu.pl}

\author[J. Milewski]{Jan Milewski}
\address{Institute of Mathematics, Pozna\'n University of Technology,
ul. Piotrowo 3A, 60-965 Pozna\'n, Poland}
\email{jsmilew@wp.pl}

\subjclass[2010]{14D07}
\date{}
\keywords{Mixed Hodge structure, Deligne splitting, Weierstrass $\sigma$ function.}

\thanks{*Partially supported by the NCN (National Center of Science for Poland)
NN201 607440} 

\begin{abstract}
{Un $\sigma$-op{\' e}rateur sur la complexification $V_\C$ d'un espace
vectoriel r{\' e}el $V_{\R}$ est un op{\' e}rateur $A \in End_\C(V_\C)$ tel que
$\sigma(A) = 0$, o{\` u} $\sigma (z)$ est la fonction $\sigma$ de Weierstrass.
Dans cet article, nous introduisons la notion de $\sigma$-op{\' e}rateur
fortement pseudo-r{\' e}el et d{\' e}montrons qu'il y a une correspondance
biunivoque entre les structures de Hodge mixtes r{\' e}elles
et les $\sigma$-op{\' e}rateurs fortement pseudo-r{\' e}els.

\medskip

A $\sigma$-operator on a complexification $V_{\C}$
of an $\R$-vector space $V_{\R}$ is an operator $A \in \rm{End}_{\C} (V_{\C})$ such that $\sigma (A) = 0$
where $\sigma (z)$ denotes the Weierstrass $\sigma$-function.
In this paper we define the notion of the strongly pseudo-real $\sigma$-operator and prove that
there is one to one correspondence between real mixed Hodge structures and strongly pseudo-real $\sigma$-operators. 
}
\end{abstract}

\maketitle

\section{Introduction}
Let $\SSS := R_{\C/\R} \G_m$ and let $V_{\R}$ be a finite dimensional real vector space.  
By a real Hodge structure (HS) on $V_{\R}$ we understand a finite direct sum of real pure Hodge 
structures with given weights {cf. \cite{BM1}}. We can consider a Hodge structure as a real algebraic 
group representation  $\gr \, :\, \SSS \st {\rm GL}(V_{\R})$ (see eg. \cite{BM1}, 
\cite{G}, \cite{PS} for the definition of HS). 
Let
$
\mathcal{L}\rho
$
be the Lie algebra representation of $\gr$.
The following operator 
\[ S := S (\rho) :={\mathcal L}\gr (1+i) \in \End (V_{\R}),\]
introduced in \cite{BM1}, will be called the Hodge-Lie operator of the real ${\rm{HS}}$
given by $\rho.$

Let $\sigma (z) $ be the Weierstrass' sigma function for the lattice 
$\Lambda := \Z \, \omega_1 + \Z \, \omega_2$ 
with $\omega_1 = 1-i$ and $\omega_2=1+i:$
{\large \begin{equation}
\sigma (z) := z \prod_{(p,q)\neq (0,0)}
\left( 1-\frac{z}{\lambda_{p,q}} \right)
\exp \left[\frac{z}{\lambda_{p,q}}+
\frac{1}{2}\left(\frac{z}{\lambda_{p,q}}\right)^2\right]
\nonumber\end{equation}} 
where
\[ \lambda_{p,q}:= p\omega_1 + q\omega_2.\]

\begin{definition}
Let $V_K$ ($K=\R$ resp. $K= \C$) be a finite dimensional vector space over $K$. An endomorphism
$A \in {\rm{End}} (V_K)$ is called a (real resp. complex) $\gs$-operator if $\gs (A)=0$. 
\end{definition}
\medskip

Observe that if $V_{\C} = V_{\R} \otimes_{\R} \C$ then $A$ is a $\sigma$-operator on 
$V_{\C}$ if and only if $\gk A$ is a $\sigma$-operator.  
In \cite{BM1} we obtained the following result:
\begin{theorem}\label{main theorem BM1}
{\rm{(i)}} \,\, Hodge-Lie operator $S$ of any HS is a real $\gs$-operator.\newline
{\rm{(ii)}} \,\, There is one to one correspondence between HS and the 
real $\gs$-operators. This equivalence is given 
by assigning to a HS its Hodge-Lie operator:
\[\rho \,\, \longleftrightarrow \,\, S (\rho)\] 
\end{theorem}
Christopher Deninger asked us very interesting question 
whether our approach to real HS via Weierstrass $\sigma$ function can be extended to the mixed 
Hodge structures (MHS). This paper is an affirmative answer to the C. Deninger's question
(see Definition \ref{strongly pseudo real sigma operator} and Theorem 
\ref{strongly pseudo real  A determines a MHS} below). For the definition of MHS see eg.  \cite{CKS}, \cite{D1}, \cite{D2} cf. the Definition 
\ref{def of MHS} below. In all of this paper a MHS means a real MHS. 
\medskip

Let us introduce basic definitions and state the main results 
of this paper. The proofs of these results are given in section 3. 
For every
$\lambda \in \Lambda$ define:
\be \gs_{\lambda} (z):=\frac{\gs (z)}{z - \lambda}
\label{def. of sigma lambda}.\ee
We will write \, $\gs _{m,n}(z) := \gs_{\lambda_{m,n}}(z)$ for all 
$(m,n) \in \Z \times \Z.$

\begin{definition} Let $V_{\C} := V_{\R}\otimes \C$.
A $\gs$-operator $A$ on $V_{\C} $ is called a weakly pseudo-real if
\be \gs_{r,s}(A)(\gk A-A)\gs_{p,q}(A)=0
\label{weak pseudo real condition}\ee
for all $r+s\geq p+q$.
\label{weakly pseudo real sigma operator}
\end{definition}

\begin{theorem} Every weakly pseudo-real  $\sigma$-operator $A$ determines a MHS.
\label{pseudo real  A determines a MHS}\end{theorem}

A natural question is whether there is one to one correspondence between weakly pseudo-real  $\sigma$-operators and real 
MHS on $V_{\R}.$ In general the answer is no. However, if we strengthen the condition 
(\ref{weak pseudo real condition}), then we can 
strengthen the Theorem \ref{pseudo real  A determines a MHS}.

\begin{definition} Let $V_{\C} := V_{\R}\otimes \C$.
A $\gs$-operator $A$ on $V_{\C} $ is called a strongly pseudo-real if
\be \gs_{r,s}(A)(\gk A-A)\gs_{p,q}(A)=0
\label{strongly pseudo real condition}\ee
for all $r\geq p$ or $s\geq q$.
\label{strongly pseudo real sigma operator}
\end{definition}

\begin{theorem} There is one to one correspondence between MHS and strongly pseudo-real  $\gs$-operators.
\label{strongly pseudo real  A determines a MHS}\end{theorem}

Observe that $A$ is a weakly pseudo-real (resp. strongly pseudo-real) $\gs$-operator if only if $\gk{A}$ is a weakly pseudo-real 
(resp. strongly pseudo-real) $\gs$-operator.

\section{Deligne splitting of a mixed Hodge structure}
\begin{definition} A real MHS on $V_{\R}$ consists of two filtrations, a finite increasing filtration on 
$V_{\R}$, the weight filtration ${W_{\bullet}}^{\R}$ and a finite decreasing filtration 
$F^{\bullet}$ on $V_{\C} := V_{\R}\otimes \C$, the Hodge filtration which induce a pure real HS of the 
weight $n$ on each graded piece
\be 
{{\rm{Gr}}}^W_n (V_{\R})= W_n^{\R}/W_{n-1}^{\R} .
\label{Gradation}\ee
\label{def of MHS}\end{definition}

\begin{theorem} {\rm{(Deligne splitting of MHS)}}. 
For any MHS there exists exactly one decomposition:
\be
V_{\C}=\bigoplus _{p,q} I^{p,q}
\label{decomposition1}\ee
such that
\be 
W_{n}=\bigoplus_{p+q \leq n} I^{p,q}, \quad F^p =\bigoplus _{{k\geq p}\atop {q \in \Z}} I^{k,q}, 
\label{decomposition2}\ee 
\be \label{wkdl} 
{I^{p,q}=\gk{I^{q,p}} \mod D_{p-1,q-1}}, 
\ee
where
\be 
W_n:=W_n ^{\R}\otimes_{\R} \C, \quad D_{r,s}:= \bigoplus _{{k\leq r}\atop {l \leq s}} I^{k,l}. 
\label{decomposition3}\ee 
This decomposition can be expressed via weight and Hodge filtrations in the following way: 

\begin{equation} I^{p,q}=V^p_{p+q}\cap (\gk{V^{q}_{p+q}} +\gk{U^{q-1}_{p+q-1}}), 
\end{equation}
where
\be 
V^p_n:=F^p\cap W_n , \quad  \quad U^m_n:=\sum_{j\geq 0} V^{m-j}_{n-j}.
\label{notation}\ee
\label{Deligne splitting of mixed Hodge structure}
\end{theorem}

For the proof of the Theorem \ref{Deligne splitting of mixed Hodge structure}
see \cite{D2} (cf. \cite{CKS} pp. 471-472).
Observe that $\gk {V^{q}_{n}} = \gk{F^q} \cap W_n$ because $\gk {W_n}=W_n$. Moreover 

\begin{equation}
D_{p,q} =  \overline{D_{q,p}}
\label{Dpq conj. prop.}\end{equation}
\begin{equation}
W_n = \sum_{p+q = n} D_{p,q}
\label{decomposition4}\end{equation}

\section{Mixed Hodge Structures via pseudo-real $\sigma$-operators} 

Let $A$ be a $\sigma$-operator on $V_C=V_R \otimes \C$. We get the following decomposition into eigenspaces:
\begin{equation}
V_{\C}:=\bigoplus_{p,q}I^{p,q}_A, \quad\quad I^{p,q}_A:=\{ x\in V: Ax =\gl_{p,q}x\}. 
\label{decomposition op1}\end{equation}
Certainly
\be
I^{p,q}_{\gk A} = \gk{I^{q,p}_{A}}
\label{I A and overline I A}
\ee
because $\lambda_{q,p} = \gk{\lambda}_{p,q}$.
Define the weight and Hodge filtrations and bifiltration of this $\sigma$-operator in the following way: 
\begin{equation} \label{rfd} W_n^A := \bigoplus_{p+q\leq n}I^{p,q}_A,\quad   
F^p_A := \bigoplus_{{k\geq p}\atop {q\in \Z}}I^{k,q}_A , 
\quad D_{p,q}^A :=\bigoplus_{{k\leq p}\atop {l\leq q}}I^{k,l}_A.
\end{equation}

\begin{definition}
Two $\sigma$-operators $A_1$ and $A_2$ are called weakly equivalent if they determine the same 
weight filtration and they induce the same 
homomorphism on the grading of the weight filtration (hence their difference is a homomorphism  
of weight filtration of degree $-1$).
\label{weakly equivalent}\end{definition}

\begin{lemma} Let $V_{\C}$ be a complexification of the real vector space $V_{\R}$.
A $\gs$-operator $A \in {\rm{End}}_{\C} (V_{\C})$ is weakly pseudo-real if and only if operators $A$ and $\gk{A}$ 
are weakly equivalent.
\label{A weakly pseudo equiv iff A and overline A weakly equiv}
\end{lemma}
\begin{proof}
Consider the decomposition (\ref{decomposition op1}). The lemma
is a consequence of the following equality: 
\[ \gs_{p,q}(A) =\gs_{p,q}(\gl _{p,q}) P_{p,q}, \]
where $P_{p,q}$ denotes the projection operator
\[P_{p,q} : \bigoplus_{m,n} I^{m,n} _A \rightarrow I^{p,q} _A \] 
onto the direct summand $I^{p,q}_A.$  
\end{proof}

\noindent 
{\it Proof of the Theorem} \ref{pseudo real  A determines a MHS}. \newline 
By Lemma \ref{A weakly pseudo equiv iff A and overline A weakly equiv}
the $\sigma$-operators  $A$ and $\gk{A}$ are weakly equivalent, hence they determine 
the same weight filtration  $W_{\bullet} = \overline{W_{\bullet}} \;$, where $W_{\bullet} := 
W_{\bullet}^A$ and $\gk{W_{\bullet}} = W_{\bullet}^{\gk{A}}.$
Hence the weight filtration $W_{\bullet}$ is a complexification of certain increasing filtration $W_{\bullet}^{\R}$ on $V_{\R}$.
Induced quotient operators $[A]_n$, $[\overline{A}]_n$ on $W_n/W_{n-1}$ are equal, hence are real $\gs$-operators
(a real $\sigma$-operator on $V_{\C}$ is a $\gs$-operator which is a complexification of an operator on $V_{\R}$). 
All eigenvalues of the operator $[A]_n$ have real part equal $n$, hence this operator gives a pure Hodge
structure of weight $n$. Hodge filtration of the $\gs$-operator $[A]_n$ on $W_n/W_{n-1}$ is induced by the 
Hodge filtration of $A$. Hence the weight and Hodge filtrations of the $\sigma$-operator $A$ induce
a mixed Hodge structure. 
\qed

The double increasing, finite filtration $D^{A}_{\bullet, \, \bullet}$ given in (\ref{rfd}) leads to the 
following definition:

\begin{definition}
Two $\gs$-operators $A_1$ and $A_2$ are called strongly equivalent if they determine the same double filtrations
$D^{A_1}_{\bullet, \,\bullet} = D^{A_2}_{\bullet, \, \bullet}$ and their difference is a homomorphism of this bifiltration of bidegree
$(-1,-1)$. 
\end{definition}

\begin{lemma}
 A $\gs$-operator $A$ is strongly pseudo-real if and only if operators $A$ and $\gk{A}$ 
are strongly equivalent.
\label{strongly pseudo real iff MHS}\end{lemma}
\begin{proof}
The proof is very similar to the proof of Lemma \ref{A weakly pseudo equiv iff A and overline A weakly equiv}
concerning the  weakly pseudo-real $\sigma$-operator.
\end{proof}

\noindent
{\it Proof of the Theorem} \ref{strongly pseudo real  A determines a MHS}. \newline 
Every mixed Hodge structure ($W_\bullet$, $F^\bullet$) determines uniquely the canonical Deligne's decomposition
(\ref{decomposition1}) - (\ref{decomposition3}). Hence $\gs$-operator $A$
\[
 A:=\bigoplus_{p,q} \lambda_{p,q} \, {\rm{id}}_{I^{p,q}} 
\]
is also uniquely determined by this mixed Hodge structure and is strongly pseudo-real. 
Indeed
\[
 {\gk{A}} = \bigoplus_{p,q} \gl_{p,q}\, {\rm{id}}_{\gk{I^{q,p}}}, 
\]
because $\gk{\gl}_{q,p}= \gl_{p,q}$. 
Recall that $D_{q,p}=\gk{D}_{p,q}$. 
Let $x \in D_{p,q}$. For $r \leq p, \, s \leq q,\;  k < r, \;l <s$ there exist vectors 
$\widetilde{x}_{r,s} \in \gk{{I}^{s,r}}$, 
$x_{r,s} \in  I^{r,s}$, $u_{r,s;k,l}\in I^{k,l}$ such that

\[ 
x=\sum_{{r\leq p} \atop {s\leq q}} \widetilde{x} _{r,s}\; , \quad 
\widetilde{x} _{r,s}=x_{r,s}+\sum_{{k<r} \atop {l<s}} u_{r,s;k,l}.
\]
Hence
\[ 
(\gk{A} -A)x=\sum_{{r,s}, {k,l}} (\gl_{r,s}-\gl_{k,l})u_{r,s;k,l}\in D_{p-1,q-1},
\]
where the sum $\sum_{{r,s}, {k,l}}$ is given for indices $r,s,k,l$ in the range 
$r \leq p, \, s \leq q,\;  k < r, \;l <s$. By Lemma \ref{strongly pseudo real iff MHS}
the $\gs$-operator $A$ is strongly pseudo-real.

\medskip

Conversely every $\gs$-operator $A$ determines unique decomposition of $V_{\C}$ 
into eigen-subspaces (see (\ref{decomposition op1}) and (\ref{rfd})).
When $\gs$-operator $A$ is strongly pseudo-real then, by Lemma \ref{strongly pseudo real iff MHS}, 
equalities (\ref{decomposition op1}) and (\ref{rfd}) 
give the Deligne canonical decomposition of the mixed Hodge structure determined by $A$. 
Indeed in this case for $x \in I^{p,q}_A$ we have $\overline{A} x- A x =: u \in D_{p-1,q-1}^A$.
Hence
\be
\overline{A} x=\lambda_{p,q} x+u.
\label{u without prime}
\ee
It remains to check that:  
\be
x+u' \in \overline{I_{A}^{q,p}}, 
\label{u prime}\ee 
for some $u' \in D_{p-1,q-1}^A.$ Since $A$ is strongly pseudo-real we get 
$D_{\bullet,\bullet}^A = D_{\bullet,\bullet}^{\overline A}.$ Observe that
the operator $\lambda_{p,q} \id  -\overline{A}$ with domain and target restricted to the invariant 
subspace $D_{p-1,q-1}^A$ is an automorphism because
$\lambda_{p,q}$ does not belong to the spectrum of this restriction of $A$. Hence
there exists unique $u' \in D_{p-1,q-1}^A$ such that $\lambda _{p,q} u'- \overline{A} u' = u$.
By (\ref{u without prime}) this vector $u'$ satisfies the equality (\ref{u prime}) because
of (\ref{I A and overline I A}).   
Observe that the following correspondences described in this proof: 
\[
\rm{MHS} \quad \na \quad \rm{strongly \; pseudo-real \; \;}\gs-{\rm{operator}} 
\]  
\[
\rm{strongly \; pseudo-real \; \;}\gs-{\rm{operator}} \quad  \na \quad \rm{MHS} 
\]
are inverse one of the other.
\qed

\end{document}